\documentclass[10pt]{amsart}
\usepackage{calrsfs}
\usepackage{soul}
\usepackage[greek,english]{babel}
\usepackage[iso-8859-7]{inputenc}
\usepackage{graphicx}
\usepackage{hyperref}
\usepackage{amsthm}
\usepackage{amssymb}
\usepackage{multirow}
\usepackage{tikz-cd}
\usepackage{amsmath}
\usepackage{todonotes}
\usepackage{amsbsy}
\usepackage[all]{xy}
\usepackage{changes}
\usepackage{enumerate}
\usepackage{comment}
\usepackage{mathtools}

\newtheorem{theorem}{Theorem}
\newtheorem*{theorem*}{Theorem}
\newtheorem{lemma}[theorem]{Lemma}
\newtheorem{corollary}[theorem]{Corollary}
\newtheorem{proposition}[theorem]{Proposition}
\theoremstyle{definition}

\newtheorem{definition}[theorem]{Definition}

\usepackage{changes}
\definechangesauthor[color=orange,name={Aristides Kontogeorgis}]{AK}
\definechangesauthor[color=red, name=Kostas Karagiannis]{KK}

\newcommand{\Spe}{\rm Spec }
\newcommand{\mdeg}{{\rm mdeg}}

\newcommand{\lf}{\left\lfloor}
\newcommand{\rf}{\right\rfloor}

\newcommand{\init}{{\rm in_\prec}}

\date{\today}

\title[The Relative Canonical Ideal]{The Relative Canonical Ideal of the Artin-Schreier-Kummer-Witt family of curves}

\author[H. Charalambous ]{Hara Charalambous  }
\address{
Department of Mathematics	,			
Aristotle University of Thessaloniki School of Sciences,
54124, Thessaloniki, Greece}
\email{hara@math.auth.gr}

\author[K. Karagiannis]{Kostas Karagiannis }
\address{
Department of Mathematics	,			
Aristotle University of Thessaloniki School of Sciences,
54124, Thessaloniki, Greece}
\email{kkaragia@math.auth.gr}

\author[A. Kontogeorgis]{Aristides Kontogeorgis}
\address{Department of Mathematics, National and Kapodistrian  University of Athens
Pane\-pist\-imioupolis, 15784 Athens, Greece}
\email{kontogar@math.uoa.gr}

\date \today

\makeatletter
\newcommand{\aprod}{\mathop{\operator@font \hbox{\Large$\ast$}}}
\makeatother

\begin{document}

\begin{abstract}
We study the canonical model of the Artin-Schreier-Kummer-Witt flat family of curves over a ring of mixed characteristic. We first prove the relative version of a classical theorem by Petri, then use the model proposed by Bertin-M\'ezard to construct an explicit generating set for the relative canonical ideal. As a byproduct, we obtain a combinatorial criterion for a set to generate the canonical ideal, applicable to any curve satisfying the assumptions of Petri's theorem. 
\end{abstract}

\maketitle

\section{Introduction}

\subsection{The canonical ideal}
Let $X$ be a complete, non-singular, non-hyperelliptic curve of genus $g\geq 3$ over an algebraically closed field $F$ of arbitrary characteristic. Let $\Omega_{X/F}$ denote the sheaf of holomorphic differentials on $X$ and, for $n\geq 0$, let $\Omega_{X/F}^{\otimes n}$ denote the $n-$th tensor power of $\Omega_{X/F}$. The following classical result is usually referred to in the bibliography as {\em Petri's Theorem}, even though it is due to Max Noether, Enriques and Babbage as well:
\begin{theorem}\label{Petri}
\leavevmode
\begin{enumerate}
	\item 
The canonical map 
\begin{equation*}
\phi:\mathrm{Sym}(H^0(X,\Omega_{X/F}))\rightarrow \bigoplus_{n \geq 0} 
H^0(X,\Omega_{X/F}^{\otimes n})
\end{equation*}
is surjective. 
\item The kernel $I_X$ of $\phi$ is generated by elements of degree $2$ and $3$.
\item $I_X$ is generated by elements of degree $2$ except in the following cases
\begin{enumerate}
\item $X$  is a non-singular plane quintic (in this case $g=6$).
\item $X$ is trigonal, i.e. a triple covering of $\mathbb{P}^1_F$
\end{enumerate}
\end{enumerate}
\end{theorem} 
The standard terminology for the algebro-geometric objects relevant to Petri's Theorem uses the adjective {\em canonical}: the sheaf $\Omega_{X/F}$ is the {\em canonical bundle}, the ring $\bigoplus_{n \geq 0} 
H^0(X,\Omega_{X/F}^{\otimes n})$ is the {\em canonical ring}, the map $\phi$ is the {\em canonical map} and the kernel $I_X=\ker\phi$ is the {\em canonical ideal}.\\

More details on the canonical map will be given in section \ref{sec:canonical-ideal}; for a modern treatment over a field of arbitrary characteristic we refer to the article of B. Saint-Donat \cite{Saint-Donat73}.\\

The problem of determining explicit generators for the canonical ideal has attracted interest by researchers over the years. A non-exhaustive list of techniques employed includes the use of Weierstrass semigroups \cite{Oliv2}, the theory of 
Gr\"obner bases \cite{MR3525467},
 minimal free resolutions and syzygies \cite{MR1273472}. The latter are also central to Green's conjecture, solved by Voisin in \cite{MR2157134}. The purpose of this paper is to study Petri's Theorem in the context of {\em lifts of curves} as discussed below.

\subsection{Lifts of curves}

Let $k$ be a field of prime characteristic $p>0$. {\em A lift of $k$ to characteristic 0} is the field of fractions $L$ of any integral extension of the ring of Witt vectors $W(k)$, a classical construction by Witt \cite{MR1581526}, \cite{Rabinoff2014-dy} that generalizes the $p-$adic integers $\mathbb{Z}_p=W(\mathbb{F}_p)$. In what follows the field $k$ will be assumed to be algebraically closed. Note that integral extensions of $W(k)$ are discrete valuation rings of mixed characteristic, with residue field $k$.

Now consider a projective, non-singular curve $\mathcal{X}_0$ over $k$ and let $R$ be an integral extension of $W(k)$. {\em A lift of $\mathcal{X}_0/k$ to characteristic $0$}, 
is a curve $\mathcal{X}_\eta$ over $L={\rm Quot} R$, obtained as the generic fibre of a flat family of curves $\mathcal{X}/R$ whose special fibre is $\mathcal{X}_0/k$. Such lifts have been extensively used by arithmetic geometers to reduce characteristic $p$ problems to the, much better understood, characteristic $0$ case. Maybe the idea of lifting has his origin in the approach of  
J.P.Serre in \cite{SerreMexico}, who used used these ideas, before the introduction of \'etale cohomology, in his attempt to define an appropriate cohomology theory which could solve the Weil conjectures. The lifting of an algebraic variety to characteristic zero  is unfortunately  not always possible and Serre was able to give such an example, see \cite{MR0132067}. The progress made in deformation theory by Schlessinger \cite{Sch} identified the lifting obstruction as an element in $H^2(X,T_X)$, see \cite[prop. 1.2.12]{MR2247603}, \cite[5.7 p.41]{MR2583634}.

\subsection{Lifts of curves with automorphisms}

Let $\mathcal{X}_0/k$ be a projective, non-singular curve as in the previous section. Such a curve can always lifted in characteristic zero, since the obstruction lives in the second cohomology which is always zero for curves. However, one might ask if it is possible to deform the curve together with its automorphism group, see \cite{BeMe2002}. This is not always possible, since Hurwitz's bound for the order of automorphism groups in characteristic $0$ ensures that the answer for a general group $G$ is negative, see \cite{MatignonGreen98}\cite{Nak}. In the same spirit, J. Bertin in \cite{BertinCRAS} provided an obstruction for the lifting based on the Artin representation which vanishes for cyclic groups. Note that, even in positive characteristic, the order of cyclic automorphism groups is bounded by the classical Hurwitz bound, see \cite{NakAbel}.  
The existence of such a lift for cyclic $p-$groups was conjectured by Oort in \cite{MR927980} and was laid to rest three decades later by Obus-Wewers \cite{ObusWewers} and Pop \cite{MR3194816}.\\

In the meantime, the case for $G=\mathbb{Z}/p\mathbb{Z}$ was studied by Oort himself  and Sekiguchi-Suwa \cite{MR1011987,SeSu95,SOS91}, who unified the theory of cyclic extensions of the projective line in characteristic $p$ ({\em Artin-Schreier extensions}) and that of cyclic extensions of the projective line in characteristic $0$ ({\em Kummer extensions}). The unified theory is usually referred to as {\em Kummer-Artin Schreier-Witt} theory or {\em Oort-Sekiguchi-Suwa} (OSS) theory. Using these results, Bertin-M\'ezard in \cite{BeMe2002} provided an explicit description of the affine model for the Kummer curve in terms of the affine model for the Artin-Schreier curve. Following this construction, Karanikolopoulos and the third author in \cite{KaranProc} proposed the study of the Galois module structure of the relative curve $\mathcal{X}/R$. As a byproduct, they found an explicit basis of the $R-$module of relative holomorphic differentials $H^0(\mathcal{X},\Omega_{\mathcal{X}})$, using Boseck's work \cite{Boseck} on holomorphic differentials.\\

The main result of this paper is the determination of an explicit generating set for the relative canonical ideal of the unified Kummer-Artin Schreier-Witt theory, using the Bertin-M\'ezard model and the relative basis of  \cite{KaranProc} for $1-$differentials. We conclude the introduction by giving an outline of our arguments and techniques.

\subsection{Outline}
In Section \ref{sec:canonical-ideal} we give details on the canonical map and we prove a combinatorial criterion for a subset of the canonical ideal to be a generating set. The main result of this section is Proposition \ref{dimension-criterion} which says that to check if a set $G$ of homogeneous polynomials of degree two generates the canonical ideal, it suffices to check whether $\dim_F\left( S/\langle \init(G)\rangle\right)_2\leq 3(g-1)$. The above criterion reduces the problem of finding a generating set for the canonical ideal to counting initial terms; we note that the criterion is applicable to any curve satisfying the assumptions of Petri's theorem, with the exception of plane quintics and trigonal curves.\\

In Section \ref{sec:BM-model} we formalize the lifting problem for the canonical ideal of the relative curve. First, we review the results of Bertin-M\'ezard on the explicit construction of the relative curve $\mathcal{X}/R$, then, in Theorem \ref{relative-canonical-embedding}, we define the relative canonical map and prove an analogue of Petri's Theorem for the relative curve $\mathcal{X}/R$, by constructing a diagram 
\begin{equation*}
\xymatrix{
	0 \ar[r] & I_{\mathcal{X}_\eta}\ar@{^{(}->}[r] & S_L:=L[\omega_1,\ldots,\omega_g] \ar@{->>}[r]^-{\phi_\eta} & 
	\displaystyle\bigoplus_{n=0}^\infty H^0(\mathcal{X}_\eta,\Omega_{\mathcal{X}_\eta/L}^{\otimes n}) \ar[r] & 0  
	\\
	0 \ar[r] &
	 I_{\mathcal{X}}\ar@{^{(}->}[r] \ar@{^{(}->}[u]_{\otimes_R L}   \ar@{->>}[d]^{\otimes_R R/\mathfrak{m}}
	 & 
	S_R:=R[W_1,\ldots,W_g] \ar@{->>}[r]^-{\phi} \ar@{^{(}->}[u]_{\otimes_R L}  \ar@{->>}[d]^{\otimes_R R/\mathfrak{m}}
	& 
	\displaystyle\bigoplus_{n=0}^\infty
	H^0(\mathcal{X},\Omega_{\mathcal{X}/R}^{\otimes n}) \ar[r] \ar@{^{(}->}[u]_{\otimes_R L}  \ar@{->>}[d]^{\otimes_R R/\mathfrak{m}}
	& 0 
	\\
	0 \ar[r] & I_{\mathcal{X}_0}\ar@{^{(}->}[r] & S_k:=k[w_1,\ldots,w_g] \ar@{->>}[r]^-{\phi_0} &
	\displaystyle\bigoplus_{n=0}^\infty H^0(\mathcal{X}_0,\Omega_{\mathcal{X}_0/k}^{\otimes n}) \ar[r] & 0 
}
\end{equation*}
whose rows are exact and where each square is commutative. In Lemma  \ref{GeneralizeSpec}, we give a Nakayama-type criterion that reduces the problem of finding a generating set for the relative canonical ideal $I_\mathcal{X}$ to finding compatible generating sets for the canonical ideals on the two fibres. In short, we prove that if $G$ is a set of homogeneous polynomials in $I_\mathcal{X}$ such that $G\otimes_R L$ generates $I_{\mathcal{X}_\eta}$ and $G \otimes_R k$ generates $I_{\mathcal{X}_0}$ then $G$ generates $I_{\mathcal{X}}$.\\

In Section \ref{two-fibers} we state and prove results on the generators of the canonical ideal which are common for the two fibres.  To facilitate the counting, we set a correspondence between the variables of the polynomial ring in Petri's Theorem and a discrete set of points $A\subseteq \mathbb{Z}^2$. In Proposition  \ref{binomials-all} we find a binomial ideal contained in the canonical ideal, leading us to build the generating sets for the two fibres on sets of binomials. Further, in Proposition \ref{corresp-binoms-Mink}, we extend the correspondence between the variables and the set $A$ to a correspondence between the binomials and the Minkowski sum $A+A$ - see \cite[p. 28]{MR1311028}. The cardinality of the Minkowski sum is too big, so Proposition \ref{dimension-criterion} is not applicable to the binomials; in Definition \ref{definition-Ci} we find appropriate subsets of $A+A$ whose cardinalities are bounded by $3(g-1)$.\\

It turns out that these subsets of the Minkowski sum match exactly to the missing generators for the canonical ideals of the two fibers. The non-binomial generators differ for each fibre, as they are determined by the affine model of each curve and thus, we describe them separately, in Sections \ref{sec:generic-fibre} and \ref{sec:special-fibre}. The missing generators, given in Definition \ref{trinom-gens-generic} for the generic fibre and Definition \ref{trinom-gens-special} for the special fibre, are grouped in each case with the binomials of Definition \ref{binom-gens-all} and give the full generating sets in Theorem \ref{theorem 1} and Theorem \ref{theorem 2} respectively. The proofs of the two theorems are combinatorial in the sense that they are based on the counting criterion of Proposition  \ref{dimension-criterion}.\\

Section \ref{sec:reduction} contains the main result of this paper, Theorem \ref{main-theorem}: The generators of the canonical ideal of the relative curve are either binomials of the form
\[W_{N_1,\mu_1}W_{N'_1,\mu'_1}-W_{N_2,\mu_2}W_{N'_2,\mu'_2}\]
or polynomials of the form
\[
W_{N,\mu}W_{N',\mu'}-
W_{N'',\mu}W_{N''',\mu'''}+
\sum_{i=1}^{p-1}\sum_{j=j_{\min}(i)}^{(p-i)q}\lambda^{i-p}\binom{p}{i}c_{j,p-i}W_{N_j,\mu_i}W_{N_j',\mu_i'}.
\]
The reader will have to refer to Section \ref{sec:reduction} for the details on the indices of the variables and the coefficients. For the proof of Theorem \ref{main-theorem} we make essential use of our Nakayama-type Lemma  \ref{GeneralizeSpec} and Theorem \ref{relative-canonical-embedding}, our analogue to Petri's Theorem,  as reduction and thickening - \`a la Faltings 
\cite{MR740897} - are checked on the category of vector spaces, instead of the category of rings.\\

\section{A criterion for generators of the canonical ideal}
\label{sec:canonical-ideal}
Throughout this section, $X$ is a complete, non-singular, non-hyperelliptic curve of genus $g\geq 3$ over an algebraically closed field $F$ of arbitrary characteristic, which is neither a plane quintic nor trigonal. As in the introduction, let $\Omega_{X/F}$ denote the sheaf of holomorphic differentials on $X$ and, for $n\geq 1$,  let $\Omega_{X/F}^{\otimes n}$ be the $n-$th exterior power of $\Omega_{X/F}$; its global sections $H^0(X,\Omega_{X/F}^{\otimes n})$ form an $F-$vector space of dimension $d_{n,g}$ where
\begin{equation}\label{dng}
d_{n,g}=
\begin{cases}
g, & \text{ if } n=1\\
(2n-1)(g-1), & \text{ if } n>1 .
\end{cases}
\end{equation}
The direct sum of the $F-$vector spaces $H^0(X,\Omega_{X/F}^{\otimes n})$ is equipped with the structure of a graded ring: multiplication in $\displaystyle \bigoplus_{n\geq 0} H^0(X,\Omega_{X/F}^{\otimes n})$ is defined via
\begin{eqnarray*}
H^0(X,\Omega_{X/F}^{\otimes n})\times H^0(X,\Omega_{X/F}^{\otimes m})&\rightarrow& H^0(X,\Omega_{X/F}^{\otimes (n+m)}) \\
fdx^{\otimes n}\cdot gdx^{\otimes m} &\mapsto& fgdx^{\otimes (n+m)}.\nonumber
\end{eqnarray*}
Choosing coordinates $\omega_1,\ldots,\omega_g$ for $\mathbb{P}^{g-1}_F$ one can identify the symmetric algebra $\mathrm{Sym}(H^0(X,\Omega_{X/F}))$ of Petri's Theorem with the graded polynomial ring $S:=F[\omega_1,\ldots,\omega_g]$ and we have that
\begin{equation}\label{graded-ring-structure}
S=\bigoplus_{n\geq 1} S_n\text{ where } S_n=\{f\in S:\deg f=n \}.
\end{equation}
Choosing a basis $\mathbf{v}=\{f_1dx,\ldots,f_gdx\}$ for $H^0(X,\Omega_{X/F})$ allows us to extend the assignment $\omega_i\mapsto f_i dx$ and define a homogeneous map of graded rings
\begin{eqnarray*}
\phi: F[\omega_1,\ldots,\omega_g]&\rightarrow& \bigoplus_{n \geq 0} 
H^0(X,\Omega_{X/F}^{\otimes n})\\
\omega_1^{a_1}\cdots\omega_g^{a_g}&\mapsto& f_1^{a_1}\cdots f_g^{a_g} dx^{\otimes(a_1+\cdots a_g)}.\nonumber
\end{eqnarray*}
Note that when an emphasis on the basis $\mathbf{v}$ is desired, the map $\phi$ will be denoted by $\phi_{\mathbf{v}}$. The kernel of $\phi$, denoted by $I_X$, is a graded ideal, so that in analogy to eq. (\ref{graded-ring-structure}) we may write
\begin{equation*}
I_X=\bigoplus_{n\geq 1} (I_X)_n\text{ where } (I_X)_n=\{f\in I_X:\deg f=n \}.
\end{equation*}
In the context we are working, Petri's Theorem can be rewritten as follows:
\begin{theorem}\label{Petri1}
The canonical map $\phi$ is surjective and $I_X=\langle (I_X)_2\rangle$.\\
\end{theorem}
We fix a term order $\prec$ and note that each $f\in S$ has a unique leading term with respect to $\prec$, denoted by $\init(f)$. We define the initial ideal of $I_X$ as $\init(I_X)=\langle \init(f):f\in I_X \rangle$. If $S_n,\;(I_X)_n$ and $\init(I_X)_n$ are the $n-th$ graded pieces of $S,\;I_X$ and $\init(I_X)$ respectively, then both $(I_X)_n$ and $\init(I_X)_n$ are $F-$subspaces of $S_n$ and, since quotients commute with direct sums, we have that 
\begin{equation*}
\left(S/I\right)_n\cong S_n/I_n\text{ and } \left(S/\init(I)\right)_n\cong S_n/\init(I)_n.
\end{equation*}
The proposition below gives a criterion for a subset of the canonical ideal to be a generating set:
\begin{proposition}\label{dimension-criterion}
Let $G\subseteq I_X$ be a set of homogeneous polynomials of degree $2$ in $I_X$. If 
\[\dim_F\left( S/\langle \init(G)\rangle\right)_2\leq 3(g-1),\]
then $I_X=\langle G \rangle$.
\end{proposition}
\begin{proof}
We note that since $G\subseteq I_X,\;\langle \init(G)\rangle_2$ is a subspace of $\init(I_X)_2$. Therefore
\begin{equation}\label{initial-inclusion}
\dim_F \left(S/\init (I_X)\right)_2 = \dim_F S_2/\init (I_X)_2  \leq\dim_F S_2/\langle \init(G)\rangle_2 = \dim_F \left(S/\langle \init(G)\rangle\right)_2
\end{equation}
Moreover, by \cite{MR1363949}[Prop. 1.1] 
\begin{equation}\label{same-Hilb}
\dim_F \left(S/\init (I_X)\right)_2
=
\dim_F \left(S/I_X\right)_2 
\text{ and }
\dim_F \left(S/\langle \init(G)\rangle\right)_2
=
\dim_F \left(S/\langle G\ \rangle\right)_2.
\end{equation}
By Petri's Theorem and eq. (\ref{dng}), we have that
\begin{equation}\label{Petri-corol}
\dim_F \left(S/I_X\right)_2
= \dim_F H^0(X, \Omega_{X/F}^{\otimes 2})
=3(g-1).
\end{equation}
Combining equations (\ref{initial-inclusion}), (\ref{same-Hilb}), (\ref{Petri-corol}), and the hypothesis $\dim_F\left( S/\langle \init(G)\rangle\right)_2\leq 3(g-1)$ gives
\begin{eqnarray*}
\dim_F \big(S/I_X\big)_2 
=
\dim_F \big(S/\langle G \rangle\big)_2
&\Rightarrow&
\left(I_X\right)_2 = \langle G \rangle_2
\Rightarrow
I_X=\langle \left(I_X\right)_2 \rangle=\langle G \rangle
\end{eqnarray*}
completing the proof.
\end{proof}

\section{The canonical ideal of relative curves}
\label{sec:BM-model}

Let $k$ be an algebraically closed field of prime characteristic $\mathrm{char}(k)=p>0$. Denote by $W(k)[\zeta]$ the ring of Witt vectors over $k$ extended by a $p$-th root of unity $\zeta$ and let $\lambda=\zeta-1$. By \cite{Rabinoff2014-dy} $W(k)[\zeta]$ is a discrete valuation ring with maximal ideal $\mathfrak{m}$ and residue field isomorphic to $k$. Let $m\geq 1$ be a natural number not divisible by $p$; for any $1\leq \ell\leq p-1$ we write $m=pq-\ell$ and denote by
\begin{equation*}
 R=
 \begin{cases}
 W(k)[\zeta][[x_1,\ldots,x_q]] & \text{ if } \ell=1 \\
 W(k)[\zeta][[x_1,\ldots,x_{q-1}]] & \text{ if } \ell\neq 1
 \end{cases}
 \end{equation*} 
the Oort-Sekiguch-Suwa factor of the versal deformation ring as in \cite[sec. 3]{KaranProc}, which is a local ring with maximal ideal $\mathfrak{m}_R=\langle \mathfrak{m},\{x_i\}\rangle$. We write 
\begin{equation*}
K=\mathrm{Quot}\left(R/\mathfrak{m}\right)=
 \begin{cases}
\mathrm{Quot}\left(k[[x_1,\ldots,x_q]]\right) & \text{ if } \ell=1 \\
 \mathrm{Quot}\left(k[[x_1,\ldots,x_{q-1}]] \right)& \text{ if } \ell\neq 1
 \end{cases}
 \end{equation*} 
and consider the extension of the rational function field $K(x)$ given by the affine model
\begin{equation}\label{special-fibre}
\mathcal{X}_0: X^p-X=\frac{x^{\ell}}{a(x)^p},
\end{equation}
where
\begin{equation}\label{a(x)}
a(x)=
\begin{cases}
x^q+x_1 x^{q-1}+ \cdots + + x_{q-1} x + x_q  & \text{ if } \ell=1 \\
x^q+x_1 x^{q-1}+ \cdots + x_{q-1} x & \text{ if } \ell\neq 1.
\end{cases}
\end{equation}
Bertin-M\'ezard proved in \cite[sec. 4.3]{Be-Me} that the curve of eq. (\ref{special-fibre}) lifts to a curve over $L=\mathrm{Quot}(R)$ given by the affine model
 \begin{equation} 
\label{bm-model}
\mathcal{X}_{\eta}:y^p=\lambda^p x^\ell +a(x)^p,\text{ for }y=a(x)(\lambda X+1)
\end{equation}
which is the normalization of $R[x]$ in $L(y)$. This gives rise to a family $\mathcal{X}\rightarrow \Spe R$, with special fibre $\mathcal{X}_0$ and generic fibre $\mathcal{X}_{\eta}$: 
\begin{equation}\label{classic-diagram}
\begin{tikzcd}
\mathrm{Spec}(k)\times_{\mathrm{Spec}(R)}\mathcal{X}=\mathcal{X}_0\arrow[hookrightarrow]{r}{}\arrow[rightarrow]{d}{}& 
\mathcal{X}\arrow[hookleftarrow]{r}{} \arrow[rightarrow]{d}&\mathcal{X}_\eta=\mathrm{Spec}(L)\times_{\mathrm{Spec}(R)}\mathcal{X} \arrow[rightarrow]{d}{}\\  
\mathrm{Spec}(k)\arrow[hookrightarrow]{r}{} &\mathrm{Spec}(R)\arrow[hookleftarrow]{r}{}    &\mathrm{Spec}(L)
\end{tikzcd}
\end{equation}
For $n\geq 1$, we write $\Omega_{\mathcal{X}/R}^{\otimes n}$ for the sheaf of holomorphic polydifferentials on $\mathcal{X}$. By \cite[lemma II.8.9]{MR0463157} the $R-$modules $H^0(\mathcal{X},\Omega^{\otimes n}_{\mathcal{X}/R})$ are free of rank $d_{n,g}$ for all $n\geq 1$, with $d_{n,g}$ given by eq. (\ref{dng}). We select generators $W_1,\ldots,W_g$ for the symmetric algebra $\mathrm{Sym}(H^0(\mathcal{X},\Omega_{\mathcal{X}/R}))$ and identify it with the polynomial ring $R[W_1,\ldots,W_g]$. Similarly, we identify the symmetric algebras $\mathrm{Sym}(H^0(\mathcal{X}_\eta,\Omega_{\mathcal{X}_\eta/L}))$ and $\mathrm{Sym}(H^0(\mathcal{X}_0,\Omega_{\mathcal{X}_0/k}))$ with the polynomial rings $L[\omega_1,\ldots,\omega_g]$ and $k[w_1,\ldots,w_g]$ respectively. Our next result concerns the canonical embedding of the Bertin-M\'ezard family:
\begin{theorem}\label{relative-canonical-embedding}
Diagram (\ref{classic-diagram}) induces a deformation-theoretic diagram of canonical embeddings
\begin{equation} \label{gen-diagram}
\xymatrix{
	0 \ar[r] & I_{\mathcal{X}_\eta}\ar@{^{(}->}[r] & S_L:=L[\omega_1,\ldots,\omega_g] \ar@{->>}[r]^-{\phi_\eta} & 
	\displaystyle\bigoplus_{n=0}^\infty H^0(\mathcal{X}_\eta,\Omega_{\mathcal{X}_\eta/L}^{\otimes n}) \ar[r] & 0  
	\\
	0 \ar[r] &
	 I_{\mathcal{X}}\ar@{^{(}->}[r] \ar@{^{(}->}[u]_{\otimes_R L}   \ar@{->>}[d]^{\otimes_R R/\mathfrak{m}}
	 & 
	S_R:=R[W_1,\ldots,W_g] \ar@{->>}[r]^-{\phi} \ar@{^{(}->}[u]_{\otimes_R L}  \ar@{->>}[d]^{\otimes_R R/\mathfrak{m}}
	& 
	\displaystyle\bigoplus_{n=0}^\infty
	H^0(\mathcal{X},\Omega_{\mathcal{X}/R}^{\otimes n}) \ar[r] \ar@{^{(}->}[u]_{\otimes_R L}  \ar@{->>}[d]^{\otimes_R R/\mathfrak{m}}
	& 0 
	\\
	0 \ar[r] & I_{\mathcal{X}_0}\ar@{^{(}->}[r] & S_k:=k[w_1,\ldots,w_g] \ar@{->>}[r]^-{\phi_0} &
	\displaystyle\bigoplus_{n=0}^\infty H^0(\mathcal{X}_0,\Omega_{\mathcal{X}_0/k}^{\otimes n}) \ar[r] & 0 
}
\end{equation}
where $I_\mathcal{X_\eta}:=\ker\phi_\eta,\;I_{\mathcal{X}}:=\ker\phi,\;I_{\mathcal{X}_0}=\ker\phi_0$, each row is exact and each square is commutative.
\end{theorem}
\begin{proof}
Exactness of the top and bottom row of diagram (\ref{gen-diagram}) are due to Theorem \ref{Petri}, the classical result of Enriques, Petri and M. Noether. For the middle row,
we select generators $f_1dx,\ldots,f_gdx$ for $H^0(\mathcal{X},\Omega_{\mathcal{X}/R})$ and note that the assignment $W_i\mapsto f_i dx$ gives rise to a homogeneous homomorphism of graded rings
\[
\xymatrix{
\phi:R[W_1,\ldots,W_g] \ar[r]^-
{  \phi } 
& \displaystyle \bigoplus_{n=0}^\infty
	H^0(\mathcal{X},\Omega_{\mathcal{X}/R}^{\otimes n}).
	}
\]
We prove surjectivity of $\phi$ by diagram chasing: let $r\in \displaystyle\bigoplus_{n=0}^\infty H^0(\mathcal{X},\Omega_{\mathcal{X}/R}^{\otimes n})$ and write $\overline{r}=r\otimes_R 1_{R/\mathfrak{m}}\in\displaystyle\bigoplus_{n=0}^\infty H^0(\mathcal{X_0},\Omega_{\mathcal{X_0}/k}^{\otimes n})$. Since $\phi_0$ is onto, there exists $\overline{s}\in \mathrm{Sym}(H^0(\mathcal{X}_0,\Omega_{\mathcal{X}_0/k}))$ with $\phi_0(\overline{s})=\overline{r}$. Similarly, since $\mathrm{Sym}(H^0(\mathcal{X},\Omega_{\mathcal{X}/R}))\rightarrow \mathrm{Sym}(H^0(\mathcal{X}_0,\Omega_{\mathcal{X}_0/k}))$ is onto, there exists $s\in \mathrm{Sym}(H^0(\mathcal{X},\Omega_{\mathcal{X}/R}))$ with $s\otimes_R 1_{R/\mathfrak{m}}=\overline{s}$. By construction, $\phi(s)=r$, proving that $\phi$ is onto as well.
\end{proof}
We proceed with establishing a Nakayama-type criterion for a subset of the kernel $I_\mathcal{X}$ to generate the relative canonical ideal:
\begin{lemma}\label{GeneralizeSpec}
Let $G$ be a set of homogeneous polynomials in $I_\mathcal{X}$ such that $G\otimes_R L$ generates $I_{\mathcal{X}_\eta}$ and $G \otimes_R k$ generates $I_{\mathcal{X}_0}$. Then:
\begin{enumerate}[(i)]
\item For any $n\in\mathbb{N}$, the $R-$modules $\left(S_R/\langle G\rangle\right)_n$ are free of rank $d_{n,g}$.\\
\item $I_{\mathcal{X}}=\langle G \rangle$.
\end{enumerate}
\end{lemma}
\begin{proof}
For $(i)$: Let $n\in\mathbb{N}$. Since by assumption $G\otimes_R L$ and $G \otimes_R k$ generate $I_{\mathcal{X}_\eta}$ and $I_{\mathcal{X}_0}$ respectively, we have that
\[
\left(S_R/\langle G\rangle\right)_n\otimes_R L \cong \big(S_L/I_{\mathcal{X}_\eta}\big)_n\text{ and } \left(S_R/\langle G\rangle\right)_n\otimes_R k \cong \big(S_k/I_{\mathcal{X}_0}\big)_n.
\]
By Petri's Theorem \ref{Petri} we get that
\[
\big(S_L/I_{\mathcal{X}_\eta}\big)_n\cong H^0(\mathcal{X}_{\eta},\Omega^{\otimes n}_{\mathcal{X}_{\eta}/L})
\text{ and }
\big(S_k/I_{\mathcal{X}_0}\big)_n\cong H^0(\mathcal{X}_{0},\Omega^{\otimes n}_{\mathcal{X}_{0}/k})
\]
and by eq. (\ref{dng}) 
\[
\dim_L H^0(\mathcal{X}_{\eta},\Omega^{\otimes n}_{\mathcal{X}_{\eta}/L})=
\dim_k H^0(\mathcal{X}_{0},\Omega^{\otimes n}_{\mathcal{X}_{0}/k})= d_{n,g}.
\]
The result follows from \cite[lemma II.8.9]{}.

For $(ii)$: let $s\in I_{\mathcal{X}}$ and assume for contradiction that $s\notin \langle G\rangle$. Since $s\otimes 1_L \in I_{\mathcal{X}_\eta}$ and $G\otimes_R L$ generates $I_{\mathcal{X}_\eta}$, there exist $g_i\in G$ and $s_i\in S_L$ such that $s\otimes 1_L=\sum g_i s_i\otimes 1_L$. Choosing $d\in R$ to be the gcd of the denominators of the coefficients of the $s_i$, we may clear denominators to obtain $ds\otimes 1_L=\sum g_i ds_i\otimes 1_L$, with $ds_i\in S_R$ or equivalently $ds=\sum g_ids_i$ with $ds_i\in S_R$, implying that $ds\in \langle G \rangle $. If $s\notin \langle G\rangle$, then $s$ is a torsion element of $S_R/\langle G\rangle$, with its homogeneous components being torsion elements of the free $R-$modules $\left(S_R/\langle G\rangle\right)_n$ for some $n\in\mathbb{N}$. By (i), the latter are free $R-$modules, so we conclude that if $s\notin \langle G \rangle $ then $s$ must be zero, completing the proof.
\end{proof}
Lemma \ref{GeneralizeSpec} reduces the problem of determining the generating set of the relative canonical ideal to determining compatible generating sets for the canonical ideals of the two fibers. Thus, in the next section  we study the canonical embeddings of the two fibers, while compatibility is studied in section  \ref{sec:reduction}.

\section{The Canonical Embedding of the Two Fibers}
\label{two-fibers}
The family's generic fibre, given by $\mathcal{X}_{\eta}:y^p=\lambda^p x^\ell +a(x)^p$, for $y=a(x)(\lambda X+1)$, is a cyclic ramified covering of the projective line and, by assumption, the order of the cyclic group is prime to the characteristic $p$. Boseck in \cite{Boseck} gives an explicit description of a basis for the global sections of holomorphic differentials of such covers. Following the notation of \cite{KaranProc}, Boseck's basis $\mathbf{b}$ for $H^0(\mathcal{X}_\eta,\Omega_{\mathcal{X}_\eta})$ is given by
\begin{eqnarray}\label{def-Boseck}
\mathbf{b}=
\left\{
x^Ny^{-\mu}dx:\lf\frac{\mu \ell}{p}\rf\leq N\leq \mu q-2,\;1\leq\mu\leq p-1
\right\}.
\end{eqnarray}
Using this analysis, the authors of \cite{KaranProc} found an explicit basis for the global sections of holomorphic differentials on the special fibre, compatible to $\mathbf{b}$ in the sense of Lemma  \ref{GeneralizeSpec}. The basis $\overline{\mathbf{c}}$ for $H^0(\mathcal{X}_0,\Omega_{\mathcal{X}_0})$ is given by
\begin{equation}\label{def-Karan}
\overline{\mathbf{c}}=\left\{
x^N a(x)^{p-1-\mu} X^{p-1-\mu} dx:
\lf \frac{\mu\ell}{p} \rf\leq N\leq \mu q-2,\;1\leq\mu\leq p-1
\right\}.
\end{equation}
The bases $\mathbf{b}$ and $\overline{\mathbf{c}}$ are both determined by the values of $(N,\mu)$, so we proceed with the study of the respective index set.
\subsection{The index set $A$ and the corresponding multidegrees}
Let
\begin{eqnarray}\label{def-A}
A=
\left\{
(N,\mu):\lf\frac{\mu \ell}{p}\rf\leq N\leq \mu q-2,\;1\leq\mu\leq p-1
\right\}\subseteq\mathbb{N}^2.
\end{eqnarray}
and note that by \cite[eq. (34) p. 48]{Boseck}
\begin{equation}\label{genus1}
|A|=\sum_{\mu=1}^{p-1}\left(\mu q-\lf\frac{\mu \ell}{p}\rf-1\right)=g.
\end{equation}
If $\{z_{N,\mu}:(N,\mu)\in A\}$ is a set of variables indexed by $A$, to each variable $z_{N,\mu}$ we assign the multidegree $\mdeg(z_{N,\mu})=(1,N,\mu)\in\mathbb{N}^3$. Thus, if $S=F[\{z_{N,\mu}\}]$ is the polynomial ring over $F$, by assigning the multidegree $(0,0,0)$ to the elements of $F$, we get a multigrading on $S$ via
\begin{equation}\label{def-mdeg}
\mdeg(z_{N_1,\mu_1}z_{N_2,\mu_2}\cdots z_{N_d,\mu_d})=(d,N_1+N_2+\cdots N_d,\mu_1+\mu_2+\cdots+\mu_d ).
\end{equation}
We will refer to the first coordinate of the multidegree (\ref{def-mdeg}) as the \emph{standard degree}.
\\

Next, we consider the two polynomial rings $L[\{\omega_{N,\mu}\}]$ and $k[\{w_{N,\mu}\}]$ with variables indexed by the points $(N,\mu)\in A$. The results of this subsection apply to both fibers, so we introduce the following notation: We will write $X$ to refer to either curve $\mathcal{X}_\eta$ or $\mathcal{X}_0$, $F$ to refer to either field $L$ or $k$, $\{z_{N,\mu}\}$ to refer to either set of variables $\{\omega_{N,\mu}\}$ or $\{w_{N,\mu}\}$, $S:=F[\{z_{N,\mu}\}]$ to refer to either polynomial ring $L[\{\omega_{N,\mu}\}]$ or $k[\{w_{N,\mu}\}]$ and $f_{N,\mu}dx$ to refer to the basis elements of either $\mathbf{b}$ or $\overline{\mathbf{c}}$. Note that the multiplication in the canonical ring in particular implies that for any two 1-differentials $f_{N,\mu}dx,\;f_{N',\mu'}dx$ we have $f_{N,\mu}dx\cdot f_{N',\mu'}dx=f_{N+N',\mu+\mu'}dx^{\otimes 2}$.\\

\begin{definition}\label{term-order}
Let $\prec_t$ be the lexicographic order on the variables $\left\{z_{N,\mu}:(N,\mu)\in A \right\}$. We define a new term order $\prec$ on the monomials of $S$ as follows:
\begin{equation}\label{term-order1}
z_{N_1,\mu_1}z_{N_2,\mu_2}\cdots z_{N_d,\mu_d}\prec z_{N'_1,\mu'_1}z_{N'_2,\mu'_2}\cdots z_{N'_s,\mu'_s}\text{ if and only if}
\end{equation}
\begin{enumerate}[(i)]
\item $d<s$ or\\
\item $d=s$ and $\sum \mu_i >\sum \mu'_i$ or\\
\item $d=s$ and $\sum \mu_i =\sum \mu'_i$  and $\sum N_i <\sum N'_i$\\
\item $d=s$ and $\sum \mu_i =\sum \mu'_i$  and $\sum N_i =\sum N'_i$ and
\[z_{N_1,\mu_1}z_{N_2,\mu_2}\cdots z_{N_d,\mu_d}\prec_t z_{N'_1,\mu'_1}z_{N'_2,\mu'_2}\cdots z_{N'_s,\mu'_s}.\]
\end{enumerate}
\end{definition}
\subsection{The binomial part of the canonical ideal}
For each $n\in\mathbb{N}$ we write $\mathbb{T}^n$ for the set of monomials of degree $n$ in $S$ and observe that the binomials below are contained in $I_X$.
\begin{proposition}\label{binomials-all}
Let $z_{N_1,\mu_1}z_{N'_1,\mu'_1},\;z_{N_2,\mu_2}z_{N'_2,\mu'_2}\in \mathbb{T}^2$ be such that $\mdeg(z_{N_1,\mu_1}z_{N'_1,\mu'_1})=\mdeg(z_{N_2,\mu_2}z_{N'_2,\mu'_2})$. Then $z_{N_1,\mu_1}z_{N'_1,\mu'_1}-z_{N_2,\mu_2}z_{N'_2,\mu'_2}\in I_X$.
\end{proposition}
\begin{proof}
Since $\mdeg(z_{N_1,\mu_1}z_{N'_1,\mu'_1})=\mdeg(z_{N_2,\mu_2}z_{N'_2,\mu'_2})$, we have that $N_1+N'_1=N_2+N'_2$ and $\mu_1+\mu'_1=\mu_2+\mu'_2$, so
\[
\phi(z_{N_1,\mu_1}z_{N'_1,\mu'_1}-z_{N_2,\mu_2}z_{N'_2,\mu'_2})
=
f_{N_1+N'_1,\;\mu_1+\mu'_1}dx^{\otimes 2}-f_{N_2+N'_2,\;\mu_2+\mu'_2}dx^{\otimes 2}=0.
\]
\end{proof}
We collect the binomials of Proposition \ref{binomials-all} in the set below.
\begin{definition} \label{binom-gens-all}
Let
\begin{multline*}
G_1=\{
z_{N_1,\mu_1}z_{N'_1,\mu'_1}-z_{N_2,\mu_2}z_{N'_2,\mu'_2}\in S\;:
\;
z_{N_1,\mu_1}z_{N'_1,\mu'_1},z_{N_2,\mu_2}z_{N'_2,\mu'_2}\in \mathbb{T}^2\\
\text{ and }
\mdeg(z_{N_1,\mu_1}z_{N'_1,\mu'_1})=\mdeg(z_{N_2,\mu_2}z_{N'_2,\mu'_2})
\}.
\end{multline*}
\end{definition}
Next, we consider the Minkowski sum of $A$ with itself, defined as
\begin{equation*}
A+A=\{(N+N',\mu+\mu')\;:\;(N,\mu),(N',\mu')\in A\}\subseteq\mathbb{Z}^2
\end{equation*}
and note the following correspondence between points of $A+A$ and monomials in $\mathbb{T}^2$:
\begin{corollary}\label{mdeg-characterization}
\[(\rho,T)\in A+A\Leftrightarrow\exists \;z_{N,\mu}z_{N',\mu'}\in \mathbb{T}^2 \text{ such that } \mdeg(z_{N,\mu}z_{N',\mu'})=(2,\rho,T).\]
\end{corollary} 
\begin{proof}
Follows directly from the definition of $A$ given in eq. (\ref{def-A}), since
\[(N,\mu)\in A\Leftrightarrow \exists\;z_{N,\mu}\in F[\{z_{N,\mu}\}] \text{ such that }\mdeg(z_{N,\mu})=(1,N,\mu).\]
\end{proof}
The correspondence of Corollary \ref{mdeg-characterization} is not one-to-one: for any $(\rho,T)\in A+A$, we set
\begin{equation*}
B_{\rho,T}:=\{z_{N,\mu}z_{N',\mu'}\in \mathbb{T}^2\;:\; (\rho,T)=(N+N',\mu+\mu')\}
\end{equation*}
and observe that the differences of elements of $B_{\rho,T}$ are in $G_1$. Next, we define the map of sets:
\begin{definition}\label{sigma}
\begin{eqnarray*}
\sigma:A+A&\rightarrow& \mathbb{T}^2\\
(\rho,T)&\mapsto& \min_{\prec}B_{\rho,T}\nonumber
\end{eqnarray*}
\end{definition}
We will use the $\sigma$ to show that $A+A$ is in bijection with a standard basis of $\left(S / \langle \init(G_1)\rangle \right)_2$:
\begin{proposition}\label{corresp-binoms-Mink} 
$
|A+A|=\dim_F \left(S / \langle \init(G_1)\rangle \right)_2
$
\end{proposition}
\begin{proof}
Let $(\rho,T)\in A+A$. By Corollary \ref{mdeg-characterization}, $B_{\rho,T}$ is non-empty and, since $\prec$ is a total order, it has a unique minimal element. Hence, the map $\sigma$ is well-defined, $1-1$ and it is immediate that $\sigma(A+A)=\mathbb{T}^2\setminus\init(G_1)$. Since $\langle \init (G_1) \rangle $ is a monomial ideal generated in degree $2$ we remark that
$\dim_F \left(S / \langle \init(G_1)\rangle \right)_2=| \mathbb{T}^2\setminus\init(G_1) |$, completing the proof.
\end{proof}
\subsection{A subset of $A+A$ of cardinality $3(g-1)$.}
We start with the following definition:
\begin{definition}\label{lower-bound}
Let $T\in\mathbb{Z}$ such that $(\rho,T)\in A+A$. We define the quantity
\begin{equation*}
b(T)=
\begin{cases}
\lf \frac{T\ell}{p} \rf,\text{ if }\forall\;\mu,\mu'\geq 1 \text{ with }T=\mu+\mu'\text{ we have}\;\lf\frac{\mu\ell}{p}\rf + \lf\frac{\mu'\ell}{p}\rf =\lf\frac{T\ell}{p}\rf\\\\
\lf \frac{T\ell}{p}\rf -1, \text{ if }\exists\;\mu,\mu'\geq 1 \text{ with }T=\mu+\mu'\text{ and }\;\lf\frac{\mu\ell}{p}\rf + \lf\frac{\mu'\ell}{p}\rf =\lf\frac{T\ell}{p}\rf -1.
\end{cases}
\end{equation*}
\end{definition}
Definition \ref{lower-bound} allows us to give an alternative description of $A+A$:
\begin{lemma}\label{description-Minkowski}
$
A+A=\left\{(\rho,T): 2\leq T\leq  2(p-1), b(T)\leq \rho\leq Tq-4\right\}\subseteq \mathbb{N}^2
$
\end{lemma}
\begin{proof}
By definition
\[(\rho,T)\in A+A\Leftrightarrow\;\exists\; (N,\mu),(N',\mu')\in A\times A\text{ with }(\rho,T)= (N,\mu)+(N',\mu').\]
Hence, both bounds for $T$ as well as the upper bound of $\rho$ are directly given by the respective bounds for $N$ and $\mu$ in the  description of $A$ given in eq. (\ref{def-A}). The formula for $b(T)$ is deduced by the well-known property of the floor function $\lf x +y\rf -1 \leq \lf x \rf +\lf y\rf \leq \lf x +y\rf$.
\end{proof}
For $0\leq i \leq p$ we let $j_{\min}(i)$ be $0$ if $\ell=1$ and $p-i$ if $\ell\neq 1$, and consider the following subsets of $A+A$:
\begin{definition}\label{definition-Ci}
Let
\begin{multline*}
C(i)=\{
(\rho,T)\in A+A\;:\;(\rho+\ell, T+p )\text{ and }(\rho+j,T+p-i) \in A+A \\
\text{ for } j_{\min}(i)\leq j\leq (p-i)q
\}.
\end{multline*}
\end{definition}
First we study the case for $i=0$:
\begin{lemma}\label{description-C}
$
C(0)=\{(\rho,T)\in A+A\;:\;b(T)\leq \rho\leq Tq-4,\;2\leq T\leq p-2 \}.
$
\end{lemma}
\begin{proof}
By definition, for all $j_{\min}(0)\leq j \leq pq$ we have that
\begin{equation*}
(\rho,T)\in C(0)\Leftrightarrow(\rho,T)\in A+A,\;(\rho+\ell,T+p)\in A+A\text{ and }(\rho+j,T+p)\in A+A
\end{equation*}
Using Lemma \ref{description-Minkowski} we rewrite
\begin{equation*}
(\rho,T)\in C(0)\Leftrightarrow 2\leq T \leq p-2\text{ and }\max\{b(T),b(T+p)-\ell,b(T+p)-j_{\min}(0)\}\leq \rho\leq Tq-4.
\end{equation*}
We distinguish the following cases for $\max\{b(T),b(T+p)-\ell,b(T+p)-j_{\min}(0)\}$:
\begin{itemize}
\item If $\ell=1$ then $j_{\min}(0)=0$ and $b(T)=b(T+p)=0$ since $\lf \frac{\mu\ell}{p}\rf=0$ for all $1\leq\mu\leq p-1$. Hence  $\max\{b(T),b(T+p)-\ell,b(T+p)-j_{\min}(0)\}=b(T)$.\\
\item If $\ell >1$ then $j_{\min}(0)=p$, so $b(T+p)-j_{\min}(0)< b(T+p)-\ell$. Choosing an appropriate decomposition $T=\mu+\mu'$ we observe that
\[
b(T)=\lf \frac{T\ell}{p} \rf -1 \Leftrightarrow b(T+p) =\lf \frac{(T+p)\ell}{p}\rf-1.
\]
Finally, since
$
\lf \frac{(T+p)\ell}{p}\rf-\ell=\lf\frac{T\ell}{p}\rf,
$
we deduce that
$
b(T+p)-\ell= b(T)$, so that
\[\max\{b(T),b(T+p)-\ell,b(T+p)-j_{\min}(0)\}=b(T).\]
\end{itemize}
We conclude that in both cases $\max\{b(T),b(T+p)-\ell,b(T+p)-j_{\min}(0)\}=b(T)$, meaning that
\begin{equation*}
(\rho,T)\in C(0) \Leftrightarrow b(T)\leq \rho\leq Tq-4,\;2\leq T\leq p-2.
\end{equation*}
\end{proof}
We are ready to show that $G_1$ does not generate the canonical ideal:
\begin{lemma}\label{final-inequality}
$|(A+A)\setminus C(0)|\leq 3(g-1)$.
\end{lemma}
\begin{proof}
We successively have
\begin{eqnarray}\label{count-step1}
|(A+A)\setminus C(0)|&=& |A+A|-|C(0)|\nonumber \\
&=&
\sum_{T=2 }^{2(p-1)}
\left(
Tq-b(T) -3
\right)
-
\sum_{T=2 }^{p-2}
\left(
Tq-b(T) -3
\right),\text{ by Lemma  \ref{description-Minkowski} and Lemma  \ref{description-C} }\nonumber\\
&=&
\sum_{T=p-1 }^{2(p-1)}
\left(
Tq-b(T) -3
\right)\nonumber\\
&\leq&
\sum_{T=p-1 }^{2(p-1)}
\left(
Tq-\lf \frac{T\ell}{p}\rf  -2
\right)
\text{, since by Def. (\ref{lower-bound}), }b(T)\geq \lf \frac{T\ell}{p}\rf -1
\nonumber\\
&=&
\sum_{T=p+1 }^{2(p-1)}
\left(
Tq-\lf \frac{T\ell}{p}\rf  -2
\right)
+\left((p-1)q-\lf \frac{(p-1)\ell}{p}\rf  -2 \right)+\left(pq-\lf \frac{p\ell}{p}\rf  -2 \right).
\end{eqnarray}
We wish to use the relation
\begin{eqnarray}\label{genus}
\sum_{T=1}^{p-1} 
\left(
T q -\lf \frac{T\ell}{p} \rf -1
\right)=g
\end{eqnarray}
so we change the index in the sum of eq. (\ref{count-step1}) by setting $T'=T-p$:
\begin{eqnarray}\label{count-step2}
\sum_{T=p+1 }^{2(p-1)}
\left(
Tq-\lf \frac{T\ell}{p}\rf  -2
\right)
&=&
\sum_{T'=1 }^{p-2}
\left(
(T'+p)q-\lf \frac{(T'+p)\ell}{p}\rf  -2
\right)
\nonumber\\
&=&
\sum_{T'=1 }^{p-2}
\left(
T'q+pq-\lf \frac{T'\ell}{p}\rf-\ell   -2
\right)
\nonumber \\
&=&
\sum_{T'=1 }^{p-2}
\left(
T'q-\lf \frac{T'\ell}{p}\rf +m  -2
\right)
\text{, since } pq-\ell=m
\nonumber \\
&=&
\sum_{T'=1 }^{p-2}
\left(
T'q-\lf \frac{T'\ell}{p}\rf -1
\right)
+
\sum_{T'=1 }^{p-2} (m-1)
\nonumber \\
&=&
\sum_{T'=1 }^{p-2}
\left(
T'q-\lf \frac{T'\ell}{p}\rf -1
\right)
+
(m-1)(p-2).
\end{eqnarray}
Next, we observe that
\begin{equation}\label{count-step3}
\sum_{T'=1 }^{p-2}
\left(
T'q-\lf \frac{T'\ell}{p}\rf -1
\right)
+
\left((p-1)q-\lf \frac{(p-1)\ell}{p}\rf  -2 \right)
=
\sum_{T'=1 }^{p-1}
\left(
T'q-\lf \frac{T'\ell}{p}\rf -1
\right)
-1.
\end{equation}
Combining relations (\ref{count-step1}), (\ref{genus}), (\ref{count-step2}) and (\ref{count-step3}) gives:
\begin{eqnarray}\label{count-step4}
|(A+A)\setminus C(0)|
&\lneq&
\sum_{T'=1 }^{p-1}
\left(
T'q-\lf \frac{T'\ell}{p}\rf -1
\right)
-1
+ (m-1)(p-2)
+\left(pq-\lf \frac{p\ell}{p}\rf  -2 \right)\nonumber \\
&=& g-1+mp-2m-p+2+m-2 \nonumber\\
&=& g+(m-1)(p-1)-2\nonumber\\
&=&3g-2 \nonumber
\end{eqnarray}
and changing $\lneq$ to $\leq$ gives the desired
\begin{equation*}
|(A+A)\setminus C(0)|\leq 3g-3
\end{equation*}
completing the proof.
\end{proof}
We conclude this section by extending the result of Lemma \ref{final-inequality} to the intersection of the sets $C(i)$. First, we prove an auxiliary lemma:
\begin{lemma}\label{aux-b(T)}
For any $\alpha\in\mathbb{N}$, $b(T+\alpha)\leq b(T)+\alpha$.
\end{lemma}
\begin{proof}
If $\alpha=0$, the result follows trivially. For $\alpha\geq 1$, and since $\ell< p$, we have that
\[
b(T+\alpha)\leq \lf \frac{(T+\alpha)\ell}{p}\rf=
\lf \frac{T\ell}{p} +\frac{\alpha\ell}{p} \rf <
\lf \frac{T\ell}{p} +\alpha \rf =
\lf \frac{T\ell}{p}\rf +\alpha
\leq 
b(T)+1 +\alpha
\]
so that $b(T+\alpha)\leq b(T)+\alpha$.
\end{proof}
We proceed with showing that $C(0)$ is contained in $C(i)$ for all $0\leq i\leq p$:
\begin{lemma}\label{final-inequality-relative} $C(0)\subseteq C(i)$ for all $0\leq i \leq p$.
\end{lemma}
\begin{proof}
Let $(\rho,T)\in C(0)$, so that, by Lemma \ref{description-C}, $b(T)\leq \rho\leq Tq-4$ and $2\leq T\leq p-2$. To show that $(\rho,T)\in C(i)$, by Definition \ref{definition-Ci}, it suffices to show that $(\rho+j,T+p-i)\in A+A$ for $j_{\min}(i)\leq j\leq (p-i)q$. First, we observe that
\[
2\leq T\leq  T+p-i\leq p-2+p-i\leq 2(p-1)\]
and
\[
\rho+j\leq \rho +(p-i)q\leq Tq-4+(p-i)q\leq (T+p-i)q-4
\]
For the lower bound of $\rho$, we distinguish the following cases:
\begin{itemize}
\item If $\ell=1$, then $j_{\min}(i)=0$ and 
\[b(T+p-i)\leq b(T+p)\leq \rho\leq \rho + j.\]
\item If $\ell>1$ then $j_{\min}(i)=p-i$, and, by Lemma \ref{aux-b(T)}, 
\[b(T+p-i)\leq b(T)+p-i\leq \rho +p-i\leq \rho +j.\]
\end{itemize}
We conclude that $ 2\leq T \leq 2(p-1)$ and that $b(T+p-i)\leq \rho +j \leq (T+p-i)q-4$ for $j_{\min}(i)\leq j\leq (p-i)q$. Lemma \ref{description-Minkowski} implies that $(\rho+j,T+p-i)\in A+A$, completing the proof.
\end{proof}
\section{The canonical ideal on the generic fibre}
\label{sec:generic-fibre}
The affine model for family's generic fibre given in eq. (\ref{bm-model}) is equivalent to
\begin{equation}\label{generic-fibre}
\mathcal{X}_{\eta}: 1-\lambda^px^{\ell}y^{-p}-a(x)^p y^{-p}=0
\end{equation}
for $y=a(x)(\lambda X+1)$, where $a(x)$ is given by 
\begin{equation*}
a(x)=
\begin{cases}
x^q+x_1 x^{q-1}+ \cdots + + x_{q-1} x + x_q,  & \text{ if } \ell=1 \\
x^q+x_1 x^{q-1}+ \cdots + x_{q-1} x, & \text{ if } \ell\neq 1.
\end{cases}
\end{equation*}
As before, we let $j_{\min}(0)$ be $0$ if $\ell=1$ and $p$ if $\ell\neq 1$. By taking the $p$-th power of $a(x)$ we get
\begin{equation}\label{a(x)p}
a(x)^p=\sum_{j=j_{\min}(0)}^{pq} c_{j,p} x^j
\end{equation}
where for any $j_{\min}(0)\leq j\leq pq$
\begin{equation*}
c_{j,p}=
\sum_{
0 \leq t_i <p \atop t_1+2t_2 \cdots+qt_q=j}
\binom{p}{t_0,\ldots,t_q}
\prod_{i=0}^{q} x_i^{t_i}.
\end{equation*}
Let $\mathbf{b}$ be Boseck's basis for $H^0(\mathcal{X}_{\eta},\Omega_{\mathcal{X}_{\eta}/L})$ as in eq. (\ref{def-Boseck}) and let
\begin{eqnarray*}
\phi_{\eta,\mathbf{b}}:
S=L[\{\omega_{N,\mu}\}]   & \longrightarrow &
\bigoplus_{n\geq 0} H^0(\mathcal{X}_{\eta},\Omega^{\otimes n}_{\mathcal{X}_{\eta}/L}),
\\
\omega_{N_1,\mu_1}^{a_1}\cdots\omega_{N_d,\mu_d}^{a_d}
& \longmapsto &
x^{(a_1N_1+\cdots + a_d N_d)}y^{-(a_1\mu_1+\cdots + a_d\mu_d)}dx^{\otimes (a_1+\cdots +a_d)}\nonumber
\end{eqnarray*}
be the canonical map. We write $I_{\mathcal{X}_\eta}:=\ker\phi_{\eta,\mathbf{b}}$ for the canonical ideal and note that the following polynomials are in $I_{\mathcal{X}_\eta}$:
\begin{proposition}\label{trinoms-generic}
For $j_{\min}\leq j \leq pq$, let $\omega_{N,\mu}\omega_{N',\mu'},\;\omega_{N'',\mu''}\omega_{N''',\mu'''}$ and $\omega_{N_j,\mu_j}\omega_{N'_j,\mu'_j}$ be any monomials in $\mathbb{T}^2$ satisfying 
\begin{eqnarray}\label{trinoms-mdeg}
\mdeg(\omega_{N'',\mu''}\omega_{N''',\mu'''})&=&\mdeg(\omega_{N,\mu}\omega_{N',\mu'})+(0,\ell,p)\text{ and }\\
\mdeg(\omega_{N_j,\mu_j}\omega_{N'_j,\mu'_j})&=&\mdeg(\omega_{N,\mu}\omega_{N',\mu'})+(0,j,p)\nonumber.
\end{eqnarray}
Then 
\begin{equation*}
\omega_{N,\mu}\omega_{N',\mu'}
-\lambda^p\omega_{N'',\mu''}\omega_{N''',\mu'''}
-\sum_{j=j_{\text{min}}(0)}^{pq}c_{j,p}\cdot\omega_{N_j,\mu_j}\omega_{N'_j,\mu'_j}
\in I_{\mathcal{X}_{\eta}}.
\end{equation*}
\end{proposition}
\begin{proof}
Let
\begin{equation*}
f=\omega_{N,\mu}\omega_{N',\mu'}
-\lambda^p\omega_{N'',\mu''}\omega_{N''',\mu'''}
-\sum_{j=j_{\text{min}}(0)}^{pq}c_{j,p}\cdot\omega_{N_j,\mu_j}\omega_{N'_j,\mu'_j}\in S
\end{equation*}
be a polynomial whose terms satisfy the relations of eq. (\ref{trinoms-mdeg}) or, equivalently,
\begin{eqnarray}\label{trinoms-mdeg-2}
N''+N'''=N+N'+\ell &,&\;\mu''+\mu'''=\mu+\mu'+p\text{ and }\\
N_j+N'_j=N+N'+j&,&\;\mu_j+\mu'_j=\mu+\mu'+p.\nonumber
\end{eqnarray}
Applying the canonical map $\phi_{\eta,\mathbf{b}}$ to $f$ gives
\begin{equation}\label{phi-trinoms-1}
x^{N+N'}y^{-(\mu+\mu')}
dx^{\otimes 2}
-\lambda^p x^{N''+N'''}y^{-(\mu''+\mu''')}
dx^{\otimes 2}
-\sum_{j=j_{\min}(0)}^{pq}c_{j,p}\cdot x^{N_j+N'_j}y^{-(\mu_j+\mu'_j)}
dx^{\otimes 2},
\end{equation}
and using the relations of eq. (\ref{trinoms-mdeg-2}) we may rewrite eq. (\ref{phi-trinoms-1}) as
\begin{equation}\label{phi-trinoms-2}
x^{N+N'}y^{-(\mu+\mu')}
dx^{\otimes 2}
-\lambda^p x^{N+N'+\ell}y^{-(\mu+\mu'+p)}
dx^{\otimes 2}
-\sum_{j=j_{\min}(0)}^{pq}c_{j,p}\cdot x^{N+N'+j}y^{-(\mu+\mu'+p)}dx^{\otimes 2}.
\end{equation}
Factoring out $x^{N+N'}y^{-(\mu+\mu')}dx^{\otimes 2}$ from eq. (\ref{phi-trinoms-2}) gives
\begin{equation*}
x^{N+N'}y^{-(\mu+\mu')}dx^{\otimes 2}\cdot
\left(1-\lambda^px^{\ell}y^{-p}-\sum_{j=j_\text{min}(0)}^{pq}c_{j,p}x^{j}y^{-p}\right)
\end{equation*}
and combining with the expansion of $a(x)^p$ in eq. (\ref{a(x)p}) we get
\begin{equation*}
x^{N+N'}y^{-(\mu+\mu')}dx^{\otimes 2}\cdot
\left( 1-\lambda^px^{\ell}y^{-p}-a(x)^p y^{-p}\right)
\end{equation*}
which is $0$ by eq.(\ref{generic-fibre}), completing the proof.
\end{proof}
We collect the polynomials of Proposition \ref{trinoms-generic} in the set below.
\begin{definition}\label{trinom-gens-generic}
Let 
\begin{multline*}
G_2^{\mathbf{b}}:=\bigg\{
\omega_{N,\mu}\omega_{N',\mu'}
-\lambda^p\omega_{N'',\mu''}\omega_{N''',\mu'''}
-\sum_{j=j_{\text{min}}(0)}^{pq}c_{j,p}\cdot\omega_{N_j,\mu_j}\omega_{N'_j,\mu'_j}\in S\;:\\
\mdeg(\omega_{N'',\mu''}\omega_{N''',\mu'''})=\mdeg(\omega_{N,\mu}\omega_{N',\mu'})+(0,\ell,p),\\\mdeg(\omega_{N_j,\mu_j}\omega_{N'_j,\mu'_j})=\mdeg(\omega_{N,\mu}\omega_{N',\mu'})+(0,j,p),\;\\
\text{ for }j_{\min}(0)\leq j \leq pq
\;
\bigg\}.
\end{multline*}
\end{definition}
We write $G_1^{\mathbf{b}}$ for the set of binomials in Definition \ref{binom-gens-all}. The main result of this section is the following:
\begin{theorem}\label{theorem 1}
$I_{\mathcal{X}_\eta}=\left<G_1^{\mathbf{b}}\cup G_2^{\mathbf{b}} \right>$.
\end{theorem}

To prove Theorem \ref{theorem 1}, we will use the dimension criterion of Proposition \ref{dimension-criterion} and a series of lemmas. We consider the subset $C(0)$ of $A+A$ given in Definition \ref{definition-Ci}
\begin{equation*}
C(0)=\{(\rho,T)\in A+A\;|\; (\rho+\ell,T+p)\in A+A\text{ and }(\rho+j,T+p)\in A+A\text{ for }j_\text{min}(0)\leq j \leq pq \}.
\end{equation*}
and study its image under the map $ \sigma:A+A\rightarrow\mathbb{T}^2$ given in Definition \ref{sigma}.
\begin{lemma}\label{sigma-C}
$\sigma(C(0))\subseteq \init(G_2^{\mathbf{b}})$
\end{lemma}
\begin{proof}
If $(\rho,T)\in C(0)$ then by definition $(\rho,T)\in A+A$, $(\rho+\ell,T+p)\in A+A$ and $(\rho+j,T+p)\in A+A$ for all $j_\text{min}(0)\leq j \leq pq$. Hence, the monomials
\[
\omega_{N,\mu}\omega_{N',\mu'}:=\sigma(\rho,T),\;
\omega_{N'',\mu''}\omega_{N''',\mu'''}:=\sigma(\rho+\ell,T+p),\;
\omega_{N_j,\mu_j}\omega_{N'_j,\mu'_j}:=\sigma(\rho+j,T+p)
\]
give rise to a polynomial
\[g=\omega_{N,\mu}\omega_{N',\mu'}
-\lambda^p\omega_{N'',\mu''}\omega_{N''',\mu'''}
-\sum_{j=j_{\text{min}}}^{pq}c_{j,p}\cdot\omega_{N_j,\mu_j}\omega_{N'_j,\mu'_j},\]
which, by construction, satisfies $g\in G_2^{\mathbf{b}}$ and $\init(g)=\sigma(\rho,T)$.
\end{proof}
\begin{lemma}\label{corresp-trinoms-Mink}
$\dim_L \left(S/\langle \init (G_1^{\mathbf{b}}\cup G_2^{\mathbf{b}} )\rangle \right)_2\leq |(A+A)\setminus C(0)|.$
\end{lemma}
\begin{proof}
By Proposition \ref{corresp-binoms-Mink} we have that $\sigma(A+A)=\mathbb{T}^2\setminus\init(G_1^{\mathbf{b}})$ and by Lemma \ref{sigma-C} we have that $\sigma(C(0))\subseteq \init(G_2^{\mathbf{b}})$, so
\begin{equation}\label{sigma-Mink-C}
\sigma\big( (A+A)\setminus C(0)\big)\supseteq
\mathbb{T}^2\setminus\left(\init(G_1^{\mathbf{b}})\cup \init(G_2^{\mathbf{b}})\right).
\end{equation}
Since $\sigma$ is one-to-one, eq. (\ref{sigma-Mink-C}) gives
\[
|(A+A)\setminus C(0)| = |\sigma\big( (A+A)\setminus C(0)\big)|\geq
|\mathbb{T}^2\setminus\left(\init(G_1^{\mathbf{b}})\cup \init(G_2^{\mathbf{b}})\right)|.
\]
Finally, $\langle \init(G_1^{\mathbf{b}})\cup \init(G_2^{\mathbf{b}}) \rangle $ is a monomial ideal generated in degree $2$ so
\[
\dim_L \left(S / \langle \init(G_1^{\mathbf{b}})\cup \init(G_2^{\mathbf{b}}) \rangle \right)_2
=
|\mathbb{T}^2\setminus\left(\init(G_1^{\mathbf{b}})\cup \init(G_2^{\mathbf{b}})\right)|,
\]
completing the proof.
\end{proof}
\begin{proof}[Proof of Theorem \ref{theorem 1}]
By Proposition \ref{binomials-all} and Proposition \ref{trinoms-generic} we get that $\langle G_1^{\mathbf{b}}\cup G_2^{\mathbf{b}}\rangle \subseteq I_{\mathcal{X}_{\eta}}$. By Lemma \ref{corresp-trinoms-Mink} and Lemma \ref{final-inequality} we get that $\dim_L \left(S/\langle \init (G_1^{\mathbf{b}}\cup G_2^{\mathbf{b}} )\rangle \right)_2\leq 3(g-1)$. Proposition \ref{dimension-criterion} implies that $I_{\mathcal{X}_{\eta}}=\langle G_1^{\mathbf{b}}\cup G_2^{\mathbf{b}}\rangle$.
\end{proof}

\section{The canonical ideal on the special fibre}
\label{sec:special-fibre}
The affine model for the family's special fibre given in eq. (\ref{special-fibre}) is equivalent to
\begin{equation}\label{special-fibre-equiv}
\mathcal{X}_0: 1-x^{\ell}a(x)^{-p}X^{-p}-X^{-(p-1)}=0
\end{equation}
where $a(x)$ is given by eq. (\ref{a(x)}). Let $j_{\min}$ be $0$ if $\ell=1$ and $p-1$ if $\ell\neq 1$. By taking the $(p-1)$-th power of $a(x)$ we get that
\begin{equation}\label{a(x)p-1}
a(x)^{p-1}=
\sum_{j=j_{\min}(1)}^{(p-1)q} c_{j,p-1} x^j, 
\end{equation} 
where for $j_{\min}(1)\leq j \leq (p-1)q$
\begin{equation*}
c_{j,p-1}=
\sum_{(t_0,\ldots,t_q)\in\mathbb{N}^q \atop t_1+2t_2 \cdots+qt_q=j}
\binom{p-1}{t_0,\ldots,t_q}
\prod_{i=0}^{q}x_i^{t_i}.
\end{equation*}
Let $\overline{\mathbf{c}}$ be the basis for $H^0(\mathcal{X}_0,\Omega_{\mathcal{X}_0/k})$ as in eq. (\ref{def-Karan}) and let
\begin{eqnarray*}
\phi_{0,\overline{\mathbf{c}}}:S=k[\{ w_{N,\mu}\}]&\longrightarrow& \bigoplus_{n\geq 0}H^0(\mathcal{X}_0,\Omega_{\mathcal{X}_0/k}^{\otimes n})\\
w_{N_1,\mu_1}^{a_1}\cdots w_{N_d,\mu_d}^{a_d}&\longmapsto& x^{(a_1N_1+\cdots+a_dN_d)}\left(a(x)X\right)^{a_1(p-1-\mu_1)+\cdots a_d(p-1-\mu_d)}\nonumber
\end{eqnarray*}
be the canonical map. Write $I_{\mathcal{X}_0}:=\ker\phi_{0,\overline{\mathbf{c}}}$ for the canonical ideal and note that the following polynomials are in $I_{\mathcal{X}_0}$:
\begin{proposition}\label{trinoms-special}
For $j_{\min}(1)\leq j\leq (p-1)q$, let $w_{N,\mu}w_{N',\mu'},\;w_{N'',\mu''}w_{N''',\mu'''}$ and $w_{N_j,\mu_j}w_{N'_j,\mu'_j}$ be monomials in $\mathbb{T}^2$ satisfying
\begin{eqnarray}\label{mdeg-trinoms-special}
\mdeg(w_{N'',\mu''}w_{N''',\mu'''})&=&\mdeg(w_{N,\mu}w_{N',\mu'})+(0,\ell,p)\\
\mdeg(w_{N_j,\mu_j}w_{N'_j,\mu'_j})&=&\mdeg(w_{N,\mu}w_{N',\mu'})+(0,j,p-1).\nonumber
\end{eqnarray}
Then
\begin{equation*}
w_{N,\mu}w_{N',\mu'}-w_{N'',\mu''}w_{N''',\mu'''}-\sum_{j=j_{\min}(1)}^{(p-1)q}c_{j,p-1}w_{N_j,\mu_j}w_{N'_j,\mu'_j}\in I_{\mathcal{X}_0}.
\end{equation*}
\end{proposition}
\begin{proof}
Let
\begin{equation*}
f:=w_{N,\mu}w_{N',\mu'}-w_{N'',\mu''}w_{N''',\mu'''}-\sum_{j=j_{\min}(1)}^{(p-1)q}c_{j,p-1}w_{N_j,\mu_j}w_{N'_j,\mu'_j}\in S
\end{equation*}
be a polynomial whose terms satisfy the relations of eq. (\ref{mdeg-trinoms-special}) or, equivalently,
\begin{eqnarray}\label{trinoms-mdeg-special-2}
N''+N''''=N+N'+\ell&,&\mu''+\mu'''=\mu+\mu'+p\\
N_j+N'_j=N+N'+j&,&\mu_j+\mu'_j=\mu+\mu'+p-1.\nonumber
\end{eqnarray}
Applying the canonical map $\phi_{0,\overline{\mathbf{c}}}$ to $f$ gives
\begin{multline}\label{phi-trinoms-special-1}
x^{N+N'}\left(a(x)X\right)^{2p-(\mu+\mu')}dx^{\otimes 2}- x^{N''+N'''}\left(a(x)X\right)^{2p-(\mu''+\mu''')}dx^{\otimes 2}\\
-\sum_{j=j_{\min}(1)}^{(p-1)q}c_{j,p-1}x^{N_j+N'_j}\left(a(x)X\right)^{2p-(\mu_j+\mu'_j)}dx^{\otimes 2}
\end{multline}
and using the relations of eq. (\ref{trinoms-mdeg-special-2}) we may rewrite eq. (\ref{phi-trinoms-special-1}) as
\begin{multline}\label{phi-trinoms-special-2}
x^{N+N'}\left(a(x)X\right)^{2p-(\mu+\mu')}dx^{\otimes 2}- x^{N+N'+\ell}\left(a(x)X\right)^{2p-(\mu+\mu'+p)}dx^{\otimes 2}\\
-\sum_{j=j_{\min}(1)}^{(p-1)q}c_{j,p-1}x^{N+N'+j}\left(a(x)X\right)^{2p-(\mu+\mu'+p-1)}dx^{\otimes 2}
\end{multline}
Factoring out $^{N+N'}\left(a(x)X\right)^{2p-(\mu+\mu')}dx^{\otimes 2}$ from eq. (\ref{phi-trinoms-special-2}) gives
\begin{equation*}
x^{N+N'}\left(a(x)X\right)^{2p-(\mu+\mu')}dx^{\otimes 2}\cdot
\left(1-x^{\ell}\left(a(x)X\right)^{-p}-\sum_{j=j_\text{min}(1)}^{(p-1)q}c_{j,p-1} x^{j}\left(a(x)X\right)^{-(p-1)} \right)
\end{equation*}
and combining with the expansion of $a(x)^{p-1}$ in eq. (\ref{a(x)p-1}) we get
\begin{eqnarray*}
&=&x^{N+N'}\left(a(x)X\right)^{2p-(\mu+\mu')}dx^{\otimes 2}\cdot\left(1-x^{\ell}\left(a(x)X\right)^{-p}-a(x)^{p-1}\left(a(x)X\right)^{-(p-1)}\right)\\
&=&x^{N+N'}\left(a(x)X\right)^{2p-(\mu+\mu')}dx^{\otimes 2}\cdot\left(1-x^{\ell}a(x)^{-p}X^{-p}-X^{-(p-1)}\right)
\end{eqnarray*}
which is $0$ by eq. (\ref{special-fibre-equiv}), completing the proof.
\end{proof}
We collect the polynomials of Proposition (\ref{trinoms-generic}) in the set below
\begin{definition}\label{trinom-gens-special}
Let
\begin{multline*}
G_2^{\overline{\mathbf{c}}}=
\bigg\{
w_{N,\mu}w_{N',\mu'}-w_{N'',\mu''}w_{N''',\mu'''}-\sum_{j=j_{\min}(1)}^{(p-1)q}c_{j,p-1}w_{N_j,\mu_j}w_{N'_j,\mu'_j} \in S\;:\;
\\
\mdeg(w_{N'',\mu''}w_{N''',\mu'''})=\mdeg(w_{N,\mu}w_{N',\mu'})+(0,\ell,p),\;\\
\mdeg(w_{N_j,\mu_j}w_{N'_j,\mu'_j})=\mdeg(w_{N,\mu}w_{N',\mu'})+(0,j,p-1),\;\\
\text{ for } j_{\min}(1)\leq j\leq (p-1)q
\bigg\}.
\end{multline*}
\end{definition}
We write $G_1^{\overline{\mathbf{c}}}$ for the set of binomials in Definition \ref{binom-gens-all}. The main result of this section is the following:
\begin{theorem}\label{theorem 2}
 $I_{\mathcal{X}_0}=\langle G_1^{\overline{\mathbf{c}}}\cup G_2^{\overline{\mathbf{c}}} \rangle$.
\end{theorem}
To prove Theorem \ref{theorem 2} we will use the dimension criterion of Proposition \ref{dimension-criterion} and a series of lemmas. We consider the subset $C(1)$ of $A+A$ given by Definition \ref{definition-Ci}
\begin{equation*}
C(1)=\{
(\rho,T)\in A+A\;:\;(\rho+\ell, T+p )\text{ and }(\rho+j,T+p-1) \in A+A \text{ for } j_{\min}(1)\leq j\leq (p-1)q
\},
\end{equation*}
and study its image under the map $\sigma:A+A\rightarrow\mathbb{T}^2$ given in Definition \ref{sigma}.
\begin{lemma}\label{sigma-D}
$\sigma(C(1))\subseteq \init(G_2^{\overline{\mathbf{c}}})$.
\end{lemma}
\begin{proof}
If $(\rho,T)\in C(1)$ then by definition $(\rho,T)\in A+A,\;(\rho+\ell,T+p)\in A+A$ and $(\rho+j,T+p-1)\in A+A$ for all $j_{\min}(1)\leq j\leq (p-1)q$. Hence the monomials
\[
w_{N,\mu}w_{N',\mu'}:=\sigma(\rho,T),w_{N'',\mu''}w_{N''',\mu'''}:=\sigma(\rho+\ell,T+p),\;w_{N_j,\mu_j}w_{N'_j,\mu'_j}:=\sigma(\rho+j,T+p-1)
\]
give rise to a polynomial
\[
g=w_{N,\mu}w_{N',\mu'}-w_{N'',\mu''}w_{N''',\mu'''}-\sum_{j=j_{\min}(1)}^{(p-1)q}c_{j,p-1}w_{N_j,\mu_j}w_{N'_j,\mu'_j},
\]
which, by construction, satisfies $g\in G_2^{\overline{\mathbf{c}}}$ and $\init(g)=\sigma(\rho,T)$. 
\end{proof}
\begin{lemma}\label{trinom-init-char}
$\dim_k \left( S/\langle \init\left(G_1^{\overline{\mathbf{c}}}\cup G_2^{\overline{\mathbf{c}}}\right)\rangle\right)_2\leq |(A+A)\setminus C(1)|.$
\end{lemma}
\begin{proof}
By Proposition \ref{corresp-binoms-Mink} we have that $\sigma(A+A)=\mathbb{T}^2\setminus\init(G_1^{\overline{\mathbf{c}}})$ and by Lemma \ref{sigma-D} we have that $\sigma(C(1))\subseteq \init(G_2^{\overline{\mathbf{c}}})$, so
\begin{equation}\label{sigma-Mink-D}
\sigma\big( (A+A)\setminus C(1)\big)\supseteq
\mathbb{T}^2\setminus\left(\init(G_1^{\overline{\mathbf{c}}})\cup \init(G_2^{\overline{\mathbf{c}}})\right).
\end{equation}
Since $\sigma$ is one-to-one, eq. (\ref{sigma-Mink-D}) gives
\[
|(A+A)\setminus C(1)| = |\sigma\big( (A+A)\setminus C(1)\big)|\geq
|\mathbb{T}^2\setminus\left(\init(G_1^{\overline{\mathbf{c}}})\cup \init(G_2^{\overline{\mathbf{c}}})\right)|.
\]
Finally, $\langle \init(G_1^{\overline{\mathbf{c}}})\cup \init(G_2^{\overline{\mathbf{c}}}) \rangle $ is a monomial ideal generated in degree $2$ so
\[
\dim_k \left(S / \langle \init(G_1^{\overline{\mathbf{c}}})\cup \init(G_2^{\overline{\mathbf{c}}}) \rangle \right)_2
=
|\mathbb{T}^2\setminus\left(\init(G_1^{\overline{\mathbf{c}}})\cup \init(G_2^{\overline{\mathbf{c}}})\right)|,
\]
completing the proof.
\end{proof}
\begin{proof}[Proof of Theorem \ref{theorem 2}]
By Proposition \ref{binomials-all} and Proposition \ref{trinoms-special} we get that $\langle G_1^{\overline{\mathbf{c}}}\cup G_2^{\overline{\mathbf{c}}}\rangle \subseteq I_{\mathcal{X}_{0}}$. By Lemma \ref{trinom-init-char} and Lemma \ref{final-inequality-relative} we get that $\dim_k \left(S/\langle \init (G_1^{\overline{\mathbf{c}}}\cup G_2^{\overline{\mathbf{c}}} )\rangle \right)_2\leq |(A+A)\setminus C(1)|\leq |(A+A)\setminus C(0)|$ so Lemma \ref{final-inequality} gives $\dim_k \left(S/\langle \init (G_1^{\overline{\mathbf{c}}}\cup G_2^{\overline{\mathbf{c}}} )\rangle \right)_2\leq 3(g-1)$. Proposition \ref{dimension-criterion} implies that $I_{\mathcal{X}_{0}}=\langle G_1^{\overline{\mathbf{c}}}\cup G_2^{\overline{\mathbf{c}}}\rangle$, completing the proof.
\end{proof}

\section{Thickening and reduction}
\label{sec:reduction}
Let $\mathcal{X}\rightarrow \Spe R$ denote the family of curves with generic fiber given by 
\begin{equation*} 
\mathcal{X}_{\eta}:y^p=\lambda^p x^\ell +a(x)^p
\end{equation*}
and special fiber given by 
\begin{equation*}
\mathcal{X}_0: X^p-X=\frac{x^{\ell}}{a(x)^p}.
\end{equation*}
Recall that for $0\leq i \leq p$ we let $j_{\min}(i)$ be 0 if $\ell=1$ and $p-i$ if $\ell\neq 1$. By taking the $(p-i)$-th power of $a(x)$ we get that
\begin{equation}\label{a(x)(p-i)}
a(x)^{p-i}=
\sum_{j=j_{\min}(i)}^{(p-i)q} c_{j,p-i} x^j
\end{equation} 
where for $j_{\min}(i)\leq j \leq (p-i)q$
\begin{equation*}
c_{j,p-i}=
\sum_{(t_0,\ldots,t_q)\in\mathbb{N}^q \atop t_1+2t_2 \cdots+qt_q=j}
\binom{p-i}{t_0,\ldots,t_q}
\prod_{s=0}^q x_s^{t_s}.
\end{equation*}
In \cite{KaranProc} the authors prove that a basis for the free $R-$module $H^0(\mathcal{X},\Omega_{\mathcal{X}/R})$ is given by
\begin{equation*}
\mathbf{c}=\left\{\frac{x^N a(x)^{p-1-\mu}X^{p-1-\mu}}{a(x)^{p-1}(\lambda X+1)^{p-1}}dx:\lf \frac{\mu\ell}{p}\rf N\leq \mu q-2,\;1\leq\mu\leq p-1 \right\}.
\end{equation*}
Let 
\begin{eqnarray*}
\phi_{\mathbf{c}}: S=R[\{W_{N,\mu}\}]&\longrightarrow& \bigoplus_{n\geq 0} H^0(\mathcal{X},\Omega_{\mathcal{X}/R}^{\otimes n})\\
W_{N_1,\mu_1}^{a_1}\cdots W_{N_d,\mu_d}^{a_d}
&\longmapsto&
\frac{x^{(a_1N_1+\cdots a_d N_d)} (a(x)X)^{a_1(p-1-\mu_1)+\cdots a_d(p-1-\mu_d)}}{a(x)^{(a_1+\ldots+a_d)(p-1)}(\lambda X+1)^{(a_1+\ldots+a_d)(p-1)}}dx^{\otimes (a_1+\ldots+a_d)}\nonumber
\end{eqnarray*}
be the canonical map. We write $I_{\mathcal{X}}:=\ker\phi_{\mathbf{c}}$ for the canonical ideal and note that the following polynomials are in $I_{\mathcal{X}}$:
\begin{proposition}\label{trinoms-relative}
Let $1\leq i \leq p-1$. For $j_{\min}(i)\leq j \leq (p-i)q$, let $W_{N,\mu}W_{N',\mu'},\;W_{N'',\mu''}W_{N''',\mu'''}$ and $W_{N_j,\mu_i}W_{N'_j,\mu'_i}$ be any monomials in $\mathbb{T}^2$ satisfying 
\begin{eqnarray*}
\mdeg(W_{N'',\mu''}W_{N''',\mu'''})&=&\mdeg(W_{N,\mu}W_{N',\mu'})+(0,\ell,p),\;\\
\mdeg(W_{N_j,\mu_i}W_{N_j',\mu_i'})&=&\mdeg(W_{N,\mu}W_{N',\mu'})+(0,j,p-i)\nonumber.
\end{eqnarray*}
Then 
\begin{equation*}
W_{N,\mu}W_{N',\mu'}-
W_{N'',\mu}W_{N''',\mu'''}+
\sum_{i=1}^{p-1}\sum_{j=j_{\min}(i)}^{(p-i)q}\lambda^{i-p}\binom{p}{i}c_{j,p-i}W_{N_j,\mu_i}W_{N_j',\mu_i'}
\in I_{\mathcal{X}}.
\end{equation*}
\end{proposition} 
\begin{proof}
Let
\begin{equation*}
f:=W_{N,\mu}W_{N',\mu'}-
W_{N'',\mu}W_{N''',\mu'''}+
\sum_{i=1}^{p-1}\sum_{j=j_{\min}(i)}^{(p-i)q}\lambda^{i-p}\binom{p}{i}c_{j,p-i}W_{N_j,\mu_i}W_{N_j',\mu_i'}
\end{equation*}
where
\begin{eqnarray}\label{reltrinoms-mdeg-2}
N''+N'''=N+N'+\ell &,&\;\mu''+\mu'''=\mu+\mu'+p\text{ and }\\
N_j+N'_j=N+N'+j&,&\;\mu_i+\mu'_i=\mu+\mu'+p-i.\nonumber
\end{eqnarray}
We note that $f\in R[\{W_{N,\mu}\}]$, since by \cite[sec. 4.3]{BeMe2002}
\begin{equation}\label{Be-Me}
p\cdot\lambda^s \equiv
\begin{cases}
0 \mod\mathfrak{m},\;\text{for}\;-(p-1)<s<0\\
 -1\mod\mathfrak{m},\;\text{for}\;s=-(p-1),
\end{cases}
\end{equation}
which implies that $\lambda^{i-p}\binom{p}{i}\in\mathfrak{m}\subseteq\mathfrak{m}_R\subseteq R$ for all $1\leq i \leq p-1$. Applying the canonical map $\phi_{\mathbf{c}}$ to $f$ gives
\begin{multline}\label{relphi-trinoms-0}
\bigg(\frac{x^{N+N'}\left(a(x)X\right)^{2(p-1)-(\mu+\mu')}}{(a(x)(\lambda X+1))^{2(p-1)}}
dx^{\otimes 2}
- \frac{x^{N''+N'''}\left(a(x)X\right)^{2(p-1)-(\mu''+\mu''')}
}{(a(x)(\lambda X+1))^{2(p-1)}}dx^{\otimes 2}\\
+\sum_{i=1}^{p-1}\sum_{j=j_{\text{min}}}^{(p-i)q}\lambda^{i-p}\binom{p}{i}c_{j,p-i}\frac{x^{N_j+N'_j}\left(a(x)X\right)^{2(p-1)-(\mu_i+\mu'_i)}}{(a(x)(\lambda X+1))^{2(p-1)}}
dx^{\otimes 2}\bigg),
\end{multline}
and using the relations of eq. (\ref{reltrinoms-mdeg-2}) we may rewrite eq. (\ref{relphi-trinoms-0}) as
\begin{multline*}
\frac{x^{N+N'}\left(a(x)X\right)^{2(p-1)-(\mu+\mu')}}{(a(x)(\lambda X+1))^{2(p-1)}}
dx^{\otimes 2}
- \frac{x^{N+N'+\ell}\left(a(x)X\right)^{2(p-1)-(\mu+\mu'+p)}}{(a(x)(\lambda X+1))^{2(p-1)}}
dx^{\otimes 2}\\
+\sum_{i=1}^{p-1}\sum_{j=j_{\text{min}}}^{(p-i)q}\lambda^{i-p}\binom{p}{i}c_{j,p-i} \frac{x^{N+N'+j}\left(a(x)X\right)^{2(p-1)-(\mu+\mu'+p-i)}}{(a(x)(\lambda X+1))^{2(p-1)}}
dx^{\otimes 2}.
\end{multline*}
If we write
\[h:=\frac{x^{N+N'}\left(a(x)X\right)^{2(p-1)-(\mu+\mu')}}{(a(x)(\lambda X+1))^{2(p-1)}}dx^{\otimes 2},\]
then
\begin{equation*}
\phi_{\mathbf{c}}(f)=h\left(
1
- x^{\ell}\left(a(x)X\right)^{-p}
+\sum_{i=1}^{p-1}\sum_{j=j_{\text{min}}}^{(p-i)q}\lambda^{i-p}\binom{p}{i}c_{j,p-i} x^{j}\left(a(x)X\right)^{i-p}
\right).
\end{equation*}
and combining with the expansion of $a(x)^{p-i}$ in eq. (\ref{a(x)(p-i)}) we get
\begin{equation*}
\phi_{\mathbf{c}}(f)=h\left(
1
- x^{\ell}\left(a(x)X\right)^{-p}
+\sum_{i=1}^{p-1}\lambda^{i-p}\binom{p}{i}X^{i-p}
\right).
\end{equation*}
We simplify the expression as follows:
\begin{eqnarray}\label{relphi-trinoms-4}
\phi_{\mathbf{c}}(f)
&=&
h\left(1
- x^{\ell}\left(a(x)X\right)^{-p}
+\sum_{i=1}^{p-1}\lambda^{i-p}\binom{p}{i}X^{i-p}\right)=h\left(
- x^{\ell}\left(a(x)X\right)^{-p}
+\sum_{i=1}^{p}\lambda^{i-p}\binom{p}{i}X^{i-p}\right)\nonumber\\
&=&
h\left(
- x^{\ell}\left(a(x)X\right)^{-p}
-\lambda^{-p}X^{-p}+\sum_{i=0}^{p}\lambda^{i-p}\binom{p}{i}X^{i-p}\right)\nonumber\\
&=&h\left(
- x^{\ell}\left(a(x)X\right)^{-p}
-\lambda^{-p}X^{-p}+\lambda^{-p}X^{-p}(\lambda X+1)^p\right).
\end{eqnarray}
Finally, since $y=a(x)(\lambda X+1)$, eq. (\ref{relphi-trinoms-4}) is equivalent to eq.(\ref{generic-fibre}), so $\phi_{\mathbf{c}}(f)\otimes_R 1_L =0$, completing the proof.
\end{proof}
We collect the polynomials of Proposition (\ref{trinoms-relative}) in the set below:
\begin{definition}\label{relative-trinoms}
Let
\begin{multline*}
G_2^{\mathbf{c}}=\bigg\{
W_{N,\mu}W_{N',\mu'}-
W_{N'',\mu}W_{N''',\mu'''}+
\sum_{i=1}^{p-1}\sum_{j=j_{\min}(i)}^{(p-i)q}\lambda^{i-p}\binom{p}{i}c_{j,p-i}W_{N_j,\mu_i}W_{N_j',\mu_i'}\in S
\;:\;\\
\mdeg(W_{N'',\mu''}W_{N''',\mu'''})=\mdeg(W_{N,\mu}W_{N',\mu'})+(0,\ell,p),\;\\
\mdeg(W_{N_j,\mu_i}W_{N_j',\mu_i'})=\mdeg(W_{N,\mu}W_{N',\mu'})+(0,j,p-i),\;\\
\text{for } 0\leq i \leq p,\;j_{\min}(i)\leq j\leq (p-i)q
\bigg\}.
\end{multline*}
\end{definition}
We write $G_1^{\mathbf{c}}$ for the set of binomials in Definition \ref{binom-gens-all}. The main result of this section is the following: 
\begin{theorem}\label{main-theorem}
$I_\mathcal{X}=\langle G_1^{\mathbf{c}}\cup G_2^{\mathbf{c}}\rangle$.
\end{theorem}
To prove Theorem \ref{main-theorem}, we will use the Nakayama-type criterion of Lemma  \ref{GeneralizeSpec} and a series of lemmas.  We first prove compatibility with the special fibre:
\begin{lemma}\label{reduction}
$G_2^{\mathbf{c}}\otimes_R k=G_2^{\overline{\mathbf{c}}}$.
\end{lemma}
\begin{proof}
Eq. (\ref{Be-Me}) implies that in the expression
\[
\sum_{i=1}^{p-1}\sum_{j=j_{\min}}^{(p-i)q}\lambda^{i-p}\binom{p}{i}c_{j,p-i}W_{N_j,\mu_i}W_{N_j',\mu_i'}
\]
only the term for $i=1$ survives reduction, giving that
\[
\left(
\sum_{i=1}^{p-1}\sum_{j=j_{\min}}^{(p-i)q}\lambda^{i-p}\binom{p}{i}c_{j,p-i}W_{N_j,\mu_i}W_{N_j',\mu_i'}\right)
\otimes_R k =
-\sum_{j=j_{\min}}^{(p-1)q}c_{j,p-1}w_{N_j,\mu_j}w_{N'_j,\mu'_j},
\]
and equivalently
\begin{multline*}
\left(W_{N,\mu}W_{N',\mu'}-
W_{N'',\mu}W_{N''',\mu'''}+
\sum_{i=1}^{p-1}\sum_{j=j_{\min}}^{(p-i)q}\lambda^{i-p}\binom{p}{i}c_{j,p-i}W_{N_j,\mu_i}W_{N_j',\mu_i'}\right)
\otimes_R k =\\
w_{N,\mu}w_{N',\mu'}-w_{N'',\mu''}w_{N''',\mu'''}-\sum_{j=j_{\min}}^{(p-1)q}c_{j,p-1}w_{N_j,\mu_j}w_{N'_j,\mu'_j},
\end{multline*}
completing the proof.
\end{proof}
Finally, we examine compatibility with the generic fibre: Let $C(i)$ denote the subsets of $A+A$ given in Definition \ref{definition-Ci}, where $0\leq i \leq p$. By Lemma \ref{final-inequality-relative}, $C(0)\subseteq C(i)$. Thus, if $(\rho,T)\in  C(0)$ then  $(\rho,T)\in A+A,\;(\rho+\ell,T+p)\in   A+A$ and $(\rho+j,T+p-i)\in A+A$ for all $j_{\min}(i)\leq j\leq (p-i)q$. Hence the monomials
\[
W_{N,\mu}W_{N',\mu'}:=\sigma(\rho,T),W_{N'',\mu}W_{N''',\mu'''}:=\sigma(\rho+\ell,T+p),\;W_{N_j,\mu_i}W_{N_j',\mu_i'}:=\sigma(\rho+j,T+p-i)
\]
give rise to the polynomial
\[
g=W_{N,\mu}W_{N',\mu'}-
W_{N'',\mu}W_{N''',\mu'''}+
\sum_{i=1}^{p-1}\sum_{j=j_{\min}}^{(p-i)q}\lambda^{i-p}\binom{p}{i}c_{j,p-i}W_{N_j,\mu_i}W_{N_j',\mu_i'} \in  
G_2^{\mathbf{c}}.
\]
We comment that $\init(g)=\sigma(\rho,T)$.
\begin{lemma}\label{reltrinom-init-char}
$\dim_L \left( S/\langle \init\left(G_1^{\mathbf{c}}\cup G_2^{\mathbf{c}}\right)\rangle\right)_2\otimes_R L\leq |(A+A)\setminus C(0)|$.
\end{lemma}
\begin{proof}
By Proposition \ref{corresp-binoms-Mink} we have that $\sigma(A+A)=\mathbb{T}^2\setminus\left(\init(G_1^{\mathbf{c}})\otimes_R L\right)$ and by the preceding comment we have that $\sigma(C(0))\subseteq \init(G_2^{\mathbf{c}})\otimes_R L$, so
\begin{equation}\label{sigma-Mink-Ci}
\sigma\big( (A+A)\setminus C(0)\big)\supseteq
\mathbb{T}^2\setminus\left(\init(G_1^{\mathbf{c}})\otimes_R L\cup \init(G_2^{\mathbf{c}})\otimes_R L\right).
\end{equation}
Since $\sigma$ is one-to-one, eq. (\ref{sigma-Mink-Ci}) gives
\[
|(A+A)\setminus C(0)| = |\sigma\big( (A+A)\setminus C(0)\big)|\geq
|\mathbb{T}^2\setminus\left(\init(G_1^{\mathbf{c}})\otimes_R L\cup \init(G_2^{\mathbf{c}})\otimes_R L\right)|.
\]
Finally, $\langle \init(G_1^{\mathbf{c}})\otimes_R L\cup \init(G_2^{\mathbf{c}})\otimes_R L) \rangle $ is a monomial ideal generated in degree $2$ so
\[
\dim_L \left( S/\langle \init\left(G_1^{\mathbf{c}}\cup G_2^{\mathbf{c}}\right)\rangle\right)_2\otimes_R L
=
|\mathbb{T}^2\setminus\left(\init(G_1^{\mathbf{c}})\otimes_R L\cup \init(G_2^{\mathbf{c}})\otimes_R L\right)|,
\]
completing the proof.
\end{proof}

We close with the proof of Theorem \ref{main-theorem}:
\begin{proof}[Proof of Theorem \ref{main-theorem}]
By Lemma \ref{reduction} we get that $\langle \left(G_1^{\mathbf{c}}\otimes_R k\right)\cup \left(G_2^{\mathbf{c}}\otimes_R k\right)\rangle = I_{\mathcal{X}_0}$. By Proposition \ref{trinoms-relative} we have that $\langle \left(G_1^{\mathbf{c}}\otimes_R L\right)\cup \left(G_2^{\mathbf{c}}\otimes_R L\right)\rangle\subseteq I_{\mathcal{X}_\eta}$. Lemma \ref{reltrinom-init-char} and Lemma \ref{final-inequality} imply that $\dim_L \left( S/\langle \init\left(G_1^{\mathbf{c}}\cup G_2^{\mathbf{c}}\right)\rangle\right)_2\otimes_R L\leq |(A+A)\setminus C(0)|\leq 3(g-1)$, so by Proposition \ref{dimension-criterion} we have that $\langle \left(G_1^{\mathbf{c}}\otimes_R L\right)\cup \left(G_2^{\mathbf{c}}\otimes_R L\right)\rangle=I_{\mathcal{X}_\eta}$. Hence, Lemma  \ref{GeneralizeSpec} gives that $I_{\mathcal{X}}=\langle G_1^{\mathbf{c}}\cup G_2^{\mathbf{c}} \rangle$.
\end{proof}

 \def\cprime{$'$}

\end{document}